\documentclass[11pt]{article}

\usepackage{amssymb,amsmath,amsthm,enumitem,sectsty,pstricks}

\allsectionsfont{\normalsize}

\newtheorem{thm}{Theorem}

\theoremstyle{definition}

\theoremstyle{remark}

\def\BE#1{\begin{equation}\label{#1}}
\def\EE{\end{equation}}
\def\eref#1{(\ref{#1})}

\def\BEnum#1{\begin{enumerate}[label=#1,leftmargin=*,topsep=-10pt,itemsep=-3pt]}
\def\EEnum{\end{enumerate}}

\def\ov#1{\overline{#1}}
\def\sf#1{\textsf{#1}}
\def\wt#1{\widetilde{#1}}
\def\tn#1{\textnormal{#1}} 
\def\lr#1{\langle{#1}\rangle}
\def\blr#1{\big\langle{#1}\big\rangle}

\def\sm#1{\begin{small}#1\end{small}}

\def\lra{\longrightarrow}

\def\C{\mathbb C}

\def\fM{\mathfrak M}
\def\cM{\mathcal M}
\def\cN{\mathcal N}
\def\cO{\mathcal O}

\def\P{\mathbb P}
\def\R{\mathbb R}
\def\Q{\mathbb Q}
\def\cS{\mathcal S}

\def\Z{\mathbb Z}

\def\fd{\mathfrak d}
\def\ff{\mathfrak f}
\def\fs{\mathfrak s}
\def\bt{\mathbf t}
\def\bt{\mathbf t}

\def\be{\beta}
\def\ep{\epsilon}
\def\ga{\gamma}
\def\om{\omega}
\def\si{\sigma}

\def\De{\Delta}
\def\Ga{\Gamma}
\def\La{\Lambda}
\def\Om{\Omega}
\def\Si{\Sigma}
\def\Ups{\Upsilon}

\def\fo{\mathfrak o}

\def\nd{\tn{d}}
\def\ev{\tn{ev}}
\def\FS{\tn{FS}}
\def\id{\tn{id}}
\def\PD{\tn{PD}}
\def\odd{\tn{odd}}

\def\eset{\emptyset}
\def\i{\infty}
\def\prt{\partial}
\def\st{\bigstar}

\begin{document}

\thispagestyle{empty}

\title{Real Topological Recursions and WDVV Relations}
\author{Aleksey Zinger\thanks{Partially supported by NSF grants DMS 1500875 and 1901979}}
\maketitle

\begin{abstract}
\noindent
Following a question of K.~Hori at K.~Fukaya's 60th  birthday conference, 
we relate the recently established WDVV-type relations 
for real Gromov-Witten invariants to topological recursion 
relations in a real setting.
We also describe precisely the connections between the relations themselves
previously observed by A.~Alcolado.
\end{abstract}

\setcounter{section}{-1}

\section{Introduction}
\label{intro_sect}

\sf{Gromov-Witten invariants} of a compact symplectic manifold $(X,\om)$
are counts of $J$-holomorphic curves in~$X$ for an $\om$-tame almost complex
structure~$J$. 
They are governed by Kontsevich-Manin's axioms~\cite{KM}.
By these axioms, relations between homology classes on 
the Deligne-Mumford moduli space~$\ov\cM_{g,l}$ of stable nodal complex genus~$g$ 
curves with~$l$ marked points lift directly to relations between Gromov-Witten invariants.
As an example, the equality of the points in the 0-th homology of
$\ov\cM_{0,4}\!\approx\!\P^1$ represented 
by the curves on the top line of Figure~\ref{TopolRel_fig} implies
that the counts of the nodal $J$-holomorphic curves in~$X$ represented 
by the two diagrams on this line are also the~same.
Along with the \sf{Splitting Axiom}~2.2.6 in~\cite{KM}, 
this equality leads to quadratic relations between counts of 
genus~0 $J$-holomorphic curves in different \sf{degrees} (homology classes).
These relations can be expressed in terms of partial differential equations~\eref{CWDVV_e},
known as \sf{WDVV equations},
on the generating function~$\Phi_{\om}$ of~\eref{CPhidfn_e} for the genus~0 Gromov-Witten invariants
of~$(X,\om)$.
The WDVV equations, first obtained in~\cite{KM} and mathematically confirmed 
for semi-positive symplectic manifolds in~\cite{RT}, completely determine
counts of rational curves of a specified degree and passing through specified constraints
in many smooth algebraic varieties from elementary input;
see Sections~5.2 in~\cite{KM}, 10 in~\cite{RT}, and~3 in~\cite{GoPa}.
The trivial relation in $H_0(\ov\cM_{0,4})$ represented by the top line 
in Figure~\ref{TopolRel_fig} is a basic example of a \sf{tautological recursion relation}
of Theorem~\ref{Cpsi_thm}.\\

A \sf{real symplectic manifold} is a triple $(X,\om,\phi)$ consisting of 
a symplectic manifold~$(X,\om)$ and an anti-symplectic involution~$\phi$ on~$(X,\om)$,
i.e.~a diffeomorphism~$\phi$ of~$X$ with itself so that $\phi^2\!=\!\id_X$ and 
$\phi^*\om\!=\!-\om$.
The fixed locus~$X^{\phi}$ of~$\phi$ is then a Lagrangian submanifold of~$(X,\om)$.
The prototypical example is the complex projective space~$\P^n$ with the Fubini-Study symplectic
form $\om\!=\!\om_{\FS}$ and the standard complex conjugation
$$\phi\!=\!\tau_n\!:\P^n\lra\P^n, \qquad 
\tau_n\big([Z_0,\ldots,Z_n]\big)=\big[\ov{Z_0},\ldots,\ov{Z_n}\big].$$
A $J$-holomorphic curve $C$ in $(X,\om,\phi)$ is called \sf{real} if $\phi(C)\!=\!C$.
Counts of real rational $J$-holomorphic curves through generic $k$-tuples of points 
in a topological component~$Y$ of~$X^{\phi}$
and $l$~pairs of $\phi$-conjugate points in $X\!-\!X^{\phi}$ dependent only on 
the degree of the curves, $k$, and~$l$ were first defined in~\cite{Wel4} 
for compact real symplectic fourfolds and 
in~\cite{Wel6} for compact real symplectic sixfolds with~$Y$ orientable.
Building on earlier work of Fukaya-Oh-Ohta-Ono~\cite{FOOO}
and C.-C.~Liu~\cite{Melissa} on the structure and orientability of
moduli spaces of stable $J$-holomorphic maps from bordered Riemann surfaces,
Solomon~\cite{Jake} and Georgieva~\cite{Penka2} interpret Welschinger's invariants
from moduli-theoretic perspectives.
The (genus~0) \sf{real Gromov-Witten invariants} of~$(X,\om,\phi)$
are the counts of real $J$-holomorphic curves arising from these perspectives.
The two moduli-theoretic interpretations of Welschinger's invariants have led
to extensions of these counts to counts of real rational $J$-holomorphic
curves with two-dimensional conjugate constraints in sixfolds 
and to counts with even-dimensional conjugate constraints 
only in higher-dimensional real symplectic manifolds satisfying certain topological
conditions.\\

\begin{figure}
\begin{pspicture}(-.5,-.2)(10,8.2)
\psset{unit=.4cm}
\psline[linewidth=.02](10.25,18)(14,20.5)\psline[linewidth=.02](10.25,19)(14,16.5)
\pscircle*(11.75,19){.15}\pscircle*(13.25,20){.15}\pscircle*(11,18.5){.15}
\rput(11.75,19.7){\sm{1}}\rput(13.8,19.7){\sm{2}}
\pscircle*(11.75,18){.15}\pscircle*(13.25,17){.15}
\rput(11.75,17.3){\sm{3}}\rput(13.8,17.3){\sm{4}}
\rput(17,18.5){\begin{Large}$=$\end{Large}}
\psline[linewidth=.02](21.25,18)(25,20.5)\psline[linewidth=.02](21.25,19)(25,16.5)
\pscircle*(22.75,19){.15}\pscircle*(24.25,20){.15}\pscircle*(22,18.5){.15}
\rput(22.75,19.7){\sm{1}}\rput(24.8,19.7){\sm{3}}
\pscircle*(22.75,18){.15}\pscircle*(24.25,17){.15}
\rput(22.75,17.3){\sm{2}}\rput(24.8,17.3){\sm{4}}
\psline[linewidth=.05](3,13)(3,8)
\psline[linewidth=.02](2.5,12.5)(6,12.5)\psline[linewidth=.02](2.5,8.5)(6,8.5)
\pscircle*(4,12.5){.15}\pscircle*(5,12.5){.15}\pscircle*(3,10.5){.15}
\rput(4,13.2){\sm{$1^+$}}\rput(5.4,13.2){\sm{$2^+$}}\rput(2.4,10.2){\sm{$1_{\R}$}}
\pscircle*(4,8.5){.15}\pscircle*(5,8.5){.15}
\rput(4,7.8){\sm{$1^-$}}\rput(5.4,7.8){\sm{$2^-$}}
\psline[linewidth=.05](11,13)(11,8)
\psline[linewidth=.02](10.5,12.5)(14,12.5)\psline[linewidth=.02](10.5,8.5)(14,8.5)
\pscircle*(12,12.5){.15}\pscircle*(13,12.5){.15}\pscircle*(11,10.5){.15}
\rput(12,13.2){\sm{$1^+$}}\rput(13.4,13.2){\sm{$2^-$}}
\rput(10.4,10.2){\sm{$1_{\R}$}}
\pscircle*(12,8.5){.15}\pscircle*(13,8.5){.15}
\rput(12,7.8){\sm{$1^-$}}\rput(13.4,7.8){\sm{$2^+$}}
\rput(8,10.5){\begin{Large}$+$\end{Large}}
\rput(17,10.5){\begin{Large}$=$\end{Large}}
\rput(23,10.5){\begin{Large}$0$\end{Large}}
\psline[linewidth=.05](3,5)(3,0)
\psline[linewidth=.02](2.5,4.5)(6,4.5)\psline[linewidth=.02](2.5,.5)(6,.5)
\pscircle*(4,4.5){.15}\pscircle*(5,4.5){.15}
\pscircle*(3,3.5){.15}\pscircle*(3,1.5){.15}
\rput(4,5.2){\sm{$1^+$}}\rput(5.4,5.2){\sm{$2^+$}}
\rput(2.4,3.8){\sm{$3^+$}}\rput(2.4,1.8){\sm{$3^-$}}
\pscircle*(4,.5){.15}\pscircle*(5,.5){.15}
\rput(4,-.2){\sm{$1^-$}}\rput(5.4,-.2){\sm{$2^-$}}
\psline[linewidth=.05](11,5)(11,0)
\psline[linewidth=.02](10.5,4.5)(14,4.5)\psline[linewidth=.02](10.5,.5)(14,.5)
\pscircle*(12,4.5){.15}\pscircle*(13,4.5){.15}
\pscircle*(11,3.5){.15}\pscircle*(11,1.5){.15}
\rput(12,5.2){\sm{$1^+$}}\rput(13.4,5.2){\sm{$2^-$}}
\rput(10.4,3.8){\sm{$3^+$}}\rput(10.4,1.8){\sm{$3^-$}}
\pscircle*(12,.5){.15}\pscircle*(13,.5){.15}
\rput(12,-.2){\sm{$1^-$}}\rput(13.4,-.2){\sm{$2^+$}}
\rput(8,2.5){\begin{Large}$+$\end{Large}}
\rput(17,2.5){\begin{Large}$=$\end{Large}}
\psline[linewidth=.05](22,5)(22,0)
\psline[linewidth=.02](21.5,4.5)(25,4.5)\psline[linewidth=.02](21.5,.5)(25,.5)
\pscircle*(23,4.5){.15}\pscircle*(24,4.5){.15}
\pscircle*(22,3.5){.15}\pscircle*(22,1.5){.15}
\rput(23,5.2){\sm{$1^+$}}\rput(24.4,5.2){\sm{$3^+$}}
\rput(21.4,3.8){\sm{$2^+$}}\rput(21.4,1.8){\sm{$2^-$}}
\pscircle*(23,.5){.15}\pscircle*(24,.5){.15}
\rput(23,-.2){\sm{$1^-$}}\rput(24.4,-.2){\sm{$3^-$}}
\psline[linewidth=.05](30,5)(30,0)
\psline[linewidth=.02](29.5,4.5)(33,4.5)\psline[linewidth=.02](29.5,.5)(33,.5)
\pscircle*(31,4.5){.15}\pscircle*(32,4.5){.15}
\pscircle*(30,3.5){.15}\pscircle*(30,1.5){.15}
\rput(31,5.2){\sm{$1^+$}}\rput(32.4,5.2){\sm{$3^-$}}
\rput(29.4,3.8){\sm{$2^+$}}\rput(29.4,1.8){\sm{$2^-$}}
\pscircle*(31,.5){.15}\pscircle*(32,.5){.15}
\rput(31,-.2){\sm{$1^-$}}\rput(32.4,-.2){\sm{$3^+$}}
\rput(27,2.5){\begin{Large}$+$\end{Large}}
\end{pspicture}
\caption{Bordisms relations between real codimension~2 co-oriented cycles of nodal curves on 
$\ov\cM_{0,4}$, $\R\ov\cM_{0,1,2}$, and $\R\ov\cM_{0,0,3}$, respectively.}
\label{TopolRel_fig}
\end{figure}
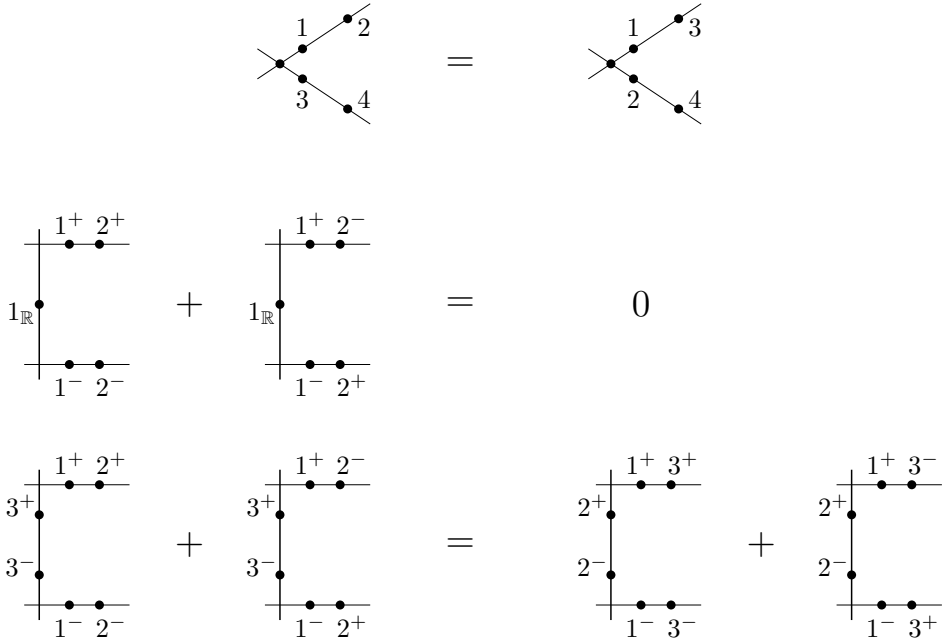

The two moduli-theoretic interpretations of Welschinger's invariants 
have also opened the door to at least potentially obtaining relations
between Welschinger's invariants by lifting relations between homology classes on
the Deligne-Mumford moduli space~$\R\ov\cM_{0,k,l}$ of stable nodal  {\it real} genus~0 
curves with $k$~real marked points and $l$~pairs $(z_i^+,z_i^-)$ of conjugate marked points.
Building on~\cite{Jake},  
Solomon's note~\cite{Jake2} predicts a pair of WDVV-type relations for
Welschinger's invariants of real symplectic fourfolds and 
the existence of similar relations for Welschinger's invariants of real symplectic sixfolds
(whenever they are defined).
Alcolado's thesis~\cite{Adam} predicts an explicit pair of WDVV-type relations for
analogues of Welschinger's invariants in much wider settings,
including those in which such invariants are yet to be defined mathematically.
A WDVV-type relation for the real Gromov-Witten invariants defined in~\cite{Penka2}
is obtained in~\cite{RealEnum} by 
establishing the relation between the one-dimensional homology classes on~$\R\ov\cM_{0,0,3}$ 
represented by the bottom line in Figure~\ref{TopolRel_fig},
lifting it to a relation between counts of real nodal $J$-holomorphic 
curves in~$(X,\om,\phi)$, and 
expressing these counts in terms of counts of real and complex irreducible curves.
The WDVV-type relations for Welschinger's invariants of real symplectic fourfolds predicted
in~\cite{Jake2} are established in~\cite{RealWDVV} by lifting relations 
between cycles in~$\R\ov\cM_{0,1,2}$ and~$\R\ov\cM_{0,0,3}$,
but over non-orientable forgetful morphisms~\eref{Cffdfn_e} from moduli spaces of real~maps.
The same general approach is used in~\cite{RealWDVVdim3} to obtain WDVV-type relations for 
Welschinger's invariants of real symplectic sixfolds with some symmetry.
As shown in~\cite{RealWDVVapp}, these relations completely determine
Welschinger's invariants of some important real symplectic fourfolds and sixfolds
from elementary input and lead to lower bounds in the real enumerative geometry 
of rational curves.  
In the case of~$(\P^3,\om_{\FS},\tau_3)$, these lower bounds fit nicely
with Koll\'ar's vanishing results~\cite{Kollar}; see Section~2.6 in~\cite{RealWDVVdim3}.\\

We recall the standard tautological recursion relation of Theorem~\ref{Cpsi_thm} and
its connection with WDVV relations between the standard (complex) Gromov-Witten invariants 
in Section~\ref{Crel_sec}.
We formulate a pair of real topological recursion relations in Theorem~\ref{Rpsi_thm}
and describe their connection with the WDVV-type relations between 
real Gromov-Witten invariants established in~\cite{RealEnum,RealWDVV,RealWDVVdim3}.
This in particular affirmatively answers a question of K.~Hori on whether  
these real WDVV-type relations arise from some real topological recursion relations.
In contrast to the relations of Theorem~\ref{Cpsi_thm}, the relations of Theorem~\ref{Rpsi_thm} 
for the generally unorientable moduli spaces~$\R\ov\cM_{0,k,l}$
involve {\it co-oriented} cycles, i.e.~cycles with orientations on the normal bundles.
We recall the standard WDVV PDEs and the two sets of real WDDV PDEs in Section~\ref{WDVV_subs}.
In Section~\ref{PDE_sec}, we show that the standard WDVV PDEs and 
one of the sets of real WDDV PDEs implies the other set when there are nonzero
real genus~0 Gromov-Witten invariants with at least~2 real marked point constraints.
The statement of Theorem~\ref{RWDVVimpl_thm} in this section and its proof 
formalize the observation and reasoning originally due to A.~Alcolado 
(as far as the author is aware);
they are included in this note for dissemination purposes.
Theorem~\ref{Rpsi_thm} is established in Section~\ref{Rrelpf_sec},
after a standard proof of Theorem~\ref{Cpsi_thm} is recalled in Section~\ref{Crelpf_sec}.\\

The author  would like to thank Xujia Chen for enlightening discussions on 
real WDVV relations,
the organizers of K.~Fukaya's 60th  birthday conference for their invitation, and
K.~Hori for raising the question answered by Theorem~\ref{Rpsi_thm} in the present note.

\section{The $\C$ case}
\label{Crel_sec}

For $l\!\in\!\Z^{\ge0}$, we define 
$$[l]=\big\{1,2,\ldots,l\big\}.$$
For $l\!\ge\!3$ and $i\!\in\![l]$, let $\psi_i\!\in\!H^2(\ov\cM_{0,l})$ 
be the first Chern class of
the universal cotangent line bundle~$L_i^*$ for the first marked point as usual.
For a partition $I\!\sqcup\!J$ of $[l]$, we denote by $D_{I,J}\!\subset\!\ov\cM_{0,l}$
the closure of the subspace of two-component curves so that one of 
the components carries the marked points indexed by~$I$ and 
the other component carries the marked points indexed by~$J$.
In particular, $D_{I,J}$ is a complex divisor in the compact complex manifold $\ov\cM_{0,l}$;
it is empty if either $|I|\!<\!2$ or $|J|\!<\!2$.

\begin{thm}[$\C$ Topological Recursion Relations]\label{Cpsi_thm}
If $i,j\!\in\![l]\!-\!\{1\}$ are distinct, then
\BE{Cpsi_e}\psi_1 =\sum_{I\sqcup J=[l]-\{1,i,j\}}
\hspace{-.3in}\PD_{\ov\cM_{0,l}}\big(D_{1I,ijJ}\big)
\in H^2\big(\ov\cM_{0,l};\Z\big).\EE
\end{thm}

\vspace{.15in}

Since $\ov\cM_{0,l}$ and $D_{1I,ijJ}$ are canonically oriented (by their complex orientations),
it is customary to drop $\PD_{\ov\cM_{0,l}}$.
For $l\!=\!4$, \eref{Cpsi_e}~gives
\BE{Cpsirel_e} D_{12,34}=\psi_1=D_{13,24}\in H_0\big(\ov\cM_{0,4};\Z\big)\,;\EE
see the top diagram in Figure~\ref{TopolRel_fig}.\\

Let $(X,\om)$ be a compact symplectic manifold, $B\!\in\!H_2(X;\Z)$,
and $J$ be an $\om$-tame almost complex structure on~$X$.
We then denote by $\ov\fM_{0,l}(B;J)$ the moduli space of stable 
degree~$B$ $J$-holomorphic maps from genus~0 connected nodal Riemann surfaces
with $l$~marked points.
For $l'\!\in\![l]\!-\![2]$, there is a natural forgetful morphism
\BE{Cffdfn_e} \ff\!:\ov\fM_{0,l}(B;J)\lra \ov\cM_{0,l'}\EE
dropping the map component and the last $l\!-\!l'$ marked points of 
each stable map.
The genus~0 degree~$B$ Gromov-Witten invariants of~$(X,\om)$ are
obtained by intersecting the~$l$ natural evaluation morphisms
from the moduli space $\ov\fM_{0,l}(B;J)$ with $l$ cycles in~$X$ in general position.\\

For $l\!\ge\!4$, the homology relation~\eref{Cpsirel_e} lifts via 
the $l'\!=\!4$ case of the morphism~\eref{Cffdfn_e} to 
a homology relation between cycles on~$\ov\fM_{0,l}(B;J)$
generically consisting of maps from two-component curves as 
on the top line in Figure~\ref{TopolRel_fig}.
The latter relation in turn translates into a relation between counts 
of two-component genus~0 degree~$B$ $J$-holomorphic curves passing through $l$~cycles. 
Replacing the node condition by the splitting of the diagonal as in~\eref{diagsplit_e}, 
one then obtains a quadratic relation between the genus~0 Gromov-Witten invariants of~$(X,\om)$.
The collection of these quadratic relations taken over all~$B$ and possible choices
of the constraints is equivalent to the PDEs~\eref{CWDVV_e} and
to the associativity of the quantum product.
However, the original proof of the associativity in~\cite{RT} instead involved defining
genus~0 Gromov-Witten invariants for each element of~$\ov\cM_{0,4}$
and showing that they are in fact independent of the chosen element
(this is more rigid than taking a path between~$D_{12,34}$ and~$D_{13,24}$
and lifting it to paths in the moduli spaces of maps).\\

Let $\fo_{D_{1I,ijJ}}^c$ denote the orientation of the normal bundle~$\cN D_{1I,ijJ}$
induced by the complex orientations of~$\ov\cM_{0,l}$ and~$D_{1I,ijJ}$.
The relation~\eref{Cpsirel_e} can then be written~as
$$ \PD_{\ov\cM_{0,4}}\big(D_{12,34},\fo_{D_{12,34}}^c\big)=\psi_1
=\PD_{\ov\cM_{0,4}}\big(D_{13,24},\fo_{D_{13,24}}^c\big)
\in H^2\big(\ov\cM_{0,4};\Z\big).$$
The equality of the left- and right-hand sides above
depends only on choosing the co-orientations $\fo_{D_{12,34}}^c$
and $\fo_{D_{12,34}}^c$ consistently, 
not on the orientations of~$\ov\cM_{0,4}$ and~$D_{1I,ijJ}$.
This form of~\eref{Cpsirel_e}
is more convenient for lifting by maps between non-orientable spaces.

\section{The $\R$ case}
\label{Rrel_sec}

For $k,l\!\in\!\Z^{\ge0}$ with $k\!+\!2l\!\ge\!2$ and $i\!\in\![l]$,
let $\psi_i\!\in\!H^2(\R\ov\cM_{0,k,l})$ be the first Chern class of
the universal cotangent line bundle~$L_i^*$ for the first marked point~$z_i^+$
in the $i$-th conjugate pair~$(z_i^+,z_i^-)$ of marked points.
For a partition $I\!\sqcup\!J$ of~$[l]$, we denote by $\R D_{I,J}\!\subset\!\R\ov\cM_{0,k,l}$
the closure of the subspace of three-component curves,
with one real component (preserved by the involution) and 
a conjugate pair of components (interchanged by the involution),
so that the conjugate components carry the conjugate pairs of 
marked points indexed by~$I$ and 
the real component carries the conjugate pairs of
marked points indexed by~$J$ (along with all real marked points).
In particular, $\R D_{I,J}$ is a compact submanifold of real codimension~2;
it is empty unless $|I|\!\ge\!2$ and either $J\!\neq\!\eset$ or $k>\!0$.
Each non-empty $\R D_{I,J}$ has precisely $2^{|I|-1}$ topological components
indexed by the unordered partitions of the marked points~$z_i^+$ with $i\!\in\!I$
between the two conjugate components of each domain curve.\\

The normal bundle $\cN\R D_{I,J}$ of $\R D_{I,J}$ in $\R\ov\cM_{0,k,l}$ consists
of conjugate pairs of smoothings of the two nodes.
We denote by~$\fo_{\R D_{I,J}}^c$ the co-orientation of~$\R D_{I,J}$ in~$\R\ov\cM_{0,k,l}$
obtained from the complex orientation of the smoothings of the node
separating off the conjugate component carrying the marked point~$z_i^+$ with 
the smallest value of~$i$ in~$I$.

\begin{thm}[$\R$ Topological Recursion Relations]\label{Rpsi_thm}
If $k\!>\!0$, then
$$\psi_1
=\sum_{I\sqcup J=[l]-\{1\}}
\hspace{-.22in}\PD_{\R\ov\cM_{0,k,l}}\big(\R D_{1I,J},\fo_{\R D_{1I,J}}^c\big)
\in H^2\big(\R\ov\cM_{0,k,l};\Z\big).$$
If $i\!\in\![l]\!-\!\{1\}$, then
$$\psi_1
=\sum_{I\sqcup J=[l]-\{1,i\}}
\hspace{-.26in}\PD_{\R\ov\cM_{0,k,l}}\big(\R D_{1I,iJ},\fo_{\R D_{1I,iJ}}^c\big)
\in H^2\big(\R\ov\cM_{0,k,l};\Z\big).$$
\end{thm}

\vspace{.15in}

For $(k,l)\!=\!(1,2)$ and $(k,l)\!=\!(0,3)$, Theorem~\ref{Rpsi_thm} gives
\begin{gather}
\label{RM12psirel_e}
\PD_{\R\ov\cM_{0,1,2}}\big(\R D_{12,\eset},\fo_{\R D_{12,\eset}}^c\big)=\psi_1=0
\in H^2\big(\R\ov\cM_{0,1,2};\Z\big),\\
\label{RM03psirel_e}
\PD_{\R\ov\cM_{0,0,3}}\big(\R D_{12,3},\fo_{\R D_{12,3}}^c\big)=\psi_1
=\PD_{\R\ov\cM_{0,0,3}}\big(\R D_{13,2},\fo_{\R D_{13,2}}^c\big)
\in H^2\big(\R\ov\cM_{0,0,3};\Z\big).
\end{gather}
The equality of the left- and right-hand sides of~\eref{RM12psirel_e}, 
depicted by the middle row in Figure~\ref{TopolRel_fig}, is immediate 
because each of the two sides of~\eref{RM12psirel_e} represents a point 
in $\R\ov\cM_{0,1,2}\!\approx\!\R\P^2$.
The equality of the left- and right-hand sides of~\eref{RM03psirel_e} in $H_1$
of the orientable space~$\R\ov\cM_{0,0,3}$, depicted by the bottom row in Figure~\ref{TopolRel_fig},
was established in~\cite{RealEnum} based
on a direct geometric study of~$\R\ov\cM_{0,0,3}$.\\

Let $(X,\om,\phi)$ be a compact real symplectic manifold.
Define
$$H_2^{\phi}(X;\Z)\equiv\big\{B'\!\in\!H_2(X;\Z)\!:\phi_*(B')\!=\!-B'\big\},\quad
H^{2*}(X;\Q)^{\phi}_{\pm}\equiv\big\{\ga\!\in\!H^{2*}(X;\Q)\!:\phi^*\ga\!=\!\pm\ga\big\}.$$
Let $B\!\in\!H_2(X;\Z)$
and $J$ be an $\om$-tame almost complex structure on~$X$ such that $\phi^*J\!=\!-J$.
We then denote by $\ov\fM_{0,k,l}^{\phi}(B;J)$ the moduli space of stable 
degree~$B$ $J$-holomorphic maps~$u$ from genus~0 connected nodal symmetric 
Riemann surfaces~$(\Si,\si)$ with $k$~real marked points $x_i\!\in\!\Si^{\si}$
and $l$~pairs of conjugate marked points \hbox{$z_i^{\pm}\!\in\!\Si\!-\!\Si^{\si}$}
such that $\phi\!\circ\!u\!=\!u\!\circ\!\si$.
This space is empty unless \hbox{$B\!\in\!H_2^{\phi}(X;\Z)$}.
For $k'\!\in\![k]$ and $l'\!\in\![l]$ with $k'\!+\!2l'\!\ge\!3$, 
there is a natural forgetful morphism
\BE{Rffdfn_e} \ff\!:\ov\fM_{0,k,l}^{\phi}(B;J)\lra \R\ov\cM_{0,k',l'}\EE
dropping the map component, the last $k\!-\!k'$ real marked points of 
each stable map, and the last $l\!-\!l'$ pairs of conjugate marked points.
The genus~0 real degree~$B$ Gromov-Witten invariants of~$(X,\om,\phi)$,
whenever defined, are obtained by intersecting 
the~$k$ evaluation morphisms from $\ov\fM_{0,k,l}^{\phi}(B;J)$ to~$X^{\phi}$ and 
the~$l$ evaluation morphisms to~$X$ (taken at~$z_i^+$)
with $k$~points in~$X^{\phi}$ and $l$~cycles in~$X$ in general position.
In contrast to the $\C$~case of Section~\ref{Crel_sec}, the total evaluation morphism
\BE{evdfn_e}\ev\!\equiv\!\big(\ev_1^{\R},\ldots,\ev_k^{\R},\ev_1^+,\ldots,\ev_l^+\big)\!:
\ov\fM_{0,k,l}^{\phi}(B;J)\lra (X^{\phi})^k\!\times\!X^l\EE
is generally not even relatively orientable in the present setting.\\

If $k\!=\!0$, $l\!\ge\!3$, and $(X,\phi)$ satisfies certain topological conditions
specified in~\cite{Penka2,Teh}, the homology relation~\eref{RM03psirel_e} 
lifts via the $(k',l')\!=\!(0,3)$ case of the morphism~\eref{Rffdfn_e} to 
a homology relation between cycles on~$\ov\fM_{0,0,l}^{\phi}(B;J)$
generically consisting of maps from three-component curves,
with one real component and a conjugate pair of components as 
on the bottom line in Figure~\ref{TopolRel_fig}.
The latter relation in turn translates into a relation between counts 
of real three-component genus~0 degree~$B$ $J$-holomorphic curves passing 
through $l$~cycles.
Replacing the node conditions by the splitting of the diagonal as in~\eref{diagsplit_e}
yields the bilinear relations between 
the real genus~0 Gromov-Witten invariants of~$(X,\om,\phi)$ with conjugate pairs of 
insertions and 
the
standard genus~0 Gromov-Witten invariants of~$(X,\om)$ obtained in~\cite{RealEnum}.
The collection of these bilinear relations taken over all~$B$ and possible choices
of the constraints is equivalent to a subset of the PDEs~\eref{WDVVodeM03_e}; 
see Section~\ref{WDVV_subs} for more details.\\

\begin{figure}
\begin{pspicture}(-.5,-.2)(10,5.6)
\psset{unit=.4cm}
\psline[linewidth=.05](6,13)(6,8)
\psline[linewidth=.02](5.5,12.5)(9,12.5)\psline[linewidth=.02](5.5,8.5)(9,8.5)
\pscircle*(7,12.5){.15}\pscircle*(8,12.5){.15}\pscircle*(6,10.5){.15}
\rput(7,13.2){\sm{$1^+$}}\rput(8.4,13.2){\sm{$2^+$}}\rput(6.7,10.5){\sm{$1_{\R}$}}
\pscircle*(7,8.5){.15}\pscircle*(8,8.5){.15}
\rput(7,7.8){\sm{$1^-$}}\rput(8.4,7.8){\sm{$2^-$}}
\psline[linewidth=.05](12,13)(12,8)
\psline[linewidth=.02](11.5,12.5)(15,12.5)\psline[linewidth=.02](11.5,8.5)(15,8.5)
\pscircle*(13,12.5){.15}\pscircle*(14,12.5){.15}\pscircle*(12,10.5){.15}
\rput(13,13.2){\sm{$1^+$}}\rput(14.4,13.2){\sm{$2^-$}}
\rput(12.7,10.5){\sm{$1_{\R}$}}
\pscircle*(13,8.5){.15}\pscircle*(14,8.5){.15}
\rput(13,7.8){\sm{$1^-$}}\rput(14.4,7.8){\sm{$2^+$}}
\rput(10.2,10.5){\begin{Large}$+$\end{Large}}
\rput(16.2,10.5){\begin{Large}$=$\end{Large}}
\pscircle[linewidth=.05](20.3,11){1.5}\pscircle[linewidth=.05](23.3,11){1.5}
\rput(17.8,10.5){\begin{large}$-2$\end{large}}
\pscircle*(21.8,11){.15}
\pscircle*(19.24,12.06){.15}\pscircle*(21.36,12.06){.15}\pscircle*(24.8,11){.15}
\pscircle*(19.24,9.94){.15}\pscircle*(21.36,9.94){.15}
\rput(18.9,12.7){\sm{$1^+$}}\rput(21.9,12.7){\sm{$2^+$}}
\rput(18.9,9.4){\sm{$1^-$}}\rput(21.9,9.4){\sm{$2^-$}}
\rput(24.1,11){\sm{$1_{\R}$}}
\rput(21.8,8){$\ep(\cS)\!\cong_4\!\!2,~\cap\!\Ups_{1,2}$}
\pscircle[linewidth=.05](29,11){1.5}\pscircle[linewidth=.05](32,11){1.5}
\rput(26.5,10.5){\begin{large}$-2$\end{large}}
\pscircle*(30.5,11){.15}
\pscircle*(30.06,12.06){.15}\pscircle*(33.06,12.06){.15}\pscircle*(27.5,11){.15}
\pscircle*(30.06,9.94){.15}\pscircle*(33.06,9.94){.15}
\rput(30.6,12.7){\sm{$1^+$}}\rput(33.6,12.7){\sm{$2^+$}}
\rput(30.6,9.4){\sm{$1^-$}}\rput(33.6,9.4){\sm{$2^-$}}
\rput(28.2,11){\sm{$1_{\R}$}}
\rput(30.5,8){$\ep(\cS)\!\cong_4\!\!2,~\cap\!\Ups_{1,2}$}
\psline[linewidth=.05](0,5)(0,0)
\psline[linewidth=.02](-.5,4.5)(3,4.5)\psline[linewidth=.02](-.5,.5)(3,.5)
\pscircle*(1,4.5){.15}\pscircle*(2,4.5){.15}
\pscircle*(0,3.5){.15}\pscircle*(0,1.5){.15}
\rput(1,5.2){\sm{$1^+$}}\rput(2.4,5.2){\sm{$2^+$}}
\rput(-.6,3.8){\sm{$3^+$}}\rput(-.6,1.8){\sm{$3^-$}}
\pscircle*(1,.5){.15}\pscircle*(2,.5){.15}
\rput(1,-.2){\sm{$1^-$}}\rput(2.4,-.2){\sm{$2^-$}}
\psline[linewidth=.05](6,5)(6,0)
\psline[linewidth=.02](5.5,4.5)(9,4.5)\psline[linewidth=.02](5.5,.5)(9,.5)
\pscircle*(7,4.5){.15}\pscircle*(8,4.5){.15}
\pscircle*(6,3.5){.15}\pscircle*(6,1.5){.15}
\rput(7,5.2){\sm{$1^+$}}\rput(8.4,5.2){\sm{$2^-$}}
\rput(5.4,3.8){\sm{$3^+$}}\rput(5.4,1.8){\sm{$3^-$}}
\pscircle*(7,.5){.15}\pscircle*(8,.5){.15}
\rput(7,-.2){\sm{$1^-$}}\rput(8.4,-.2){\sm{$2^+$}}
\rput(3.8,2.5){\begin{Large}$+$\end{Large}}
\rput(9.8,2.5){\begin{Large}$-$\end{Large}}
\psline[linewidth=.05](12,5)(12,0)
\psline[linewidth=.02](11.5,4.5)(15,4.5)\psline[linewidth=.02](11.5,.5)(15,.5)
\pscircle*(13,4.5){.15}\pscircle*(14,4.5){.15}
\pscircle*(12,3.5){.15}\pscircle*(12,1.5){.15}
\rput(13,5.2){\sm{$1^+$}}\rput(14.4,5.2){\sm{$3^+$}}
\rput(11.4,3.8){\sm{$2^+$}}\rput(11.4,1.8){\sm{$2^-$}}
\pscircle*(13,.5){.15}\pscircle*(14,.5){.15}
\rput(13,-.2){\sm{$1^-$}}\rput(14.4,-.2){\sm{$3^-$}}
\psline[linewidth=.05](18,5)(18,0)
\psline[linewidth=.02](17.5,4.5)(21,4.5)\psline[linewidth=.02](17.5,.5)(21,.5)
\pscircle*(19,4.5){.15}\pscircle*(20,4.5){.15}
\pscircle*(18,3.5){.15}\pscircle*(18,1.5){.15}
\rput(19,5.2){\sm{$1^+$}}\rput(20.4,5.2){\sm{$3^-$}}
\rput(17.4,3.8){\sm{$2^+$}}\rput(17.4,1.8){\sm{$2^-$}}
\pscircle*(19,.5){.15}\pscircle*(20,.5){.15}
\rput(19,-.2){\sm{$1^-$}}\rput(20.4,-.2){\sm{$3^+$}}
\rput(15.8,2.5){\begin{Large}$-$\end{Large}}
\rput(22.2,2.5){\begin{Large}$=$\end{Large}}
\pscircle[linewidth=.05](26.3,3){1.5}\pscircle[linewidth=.05](29.3,3){1.5}
\rput(23.8,2.5){\begin{large}$-2$\end{large}}
\pscircle*(27.8,3){.15}
\pscircle*(25.24,4.06){.15}\pscircle*(27.36,4.06){.15}\pscircle*(30.36,4.06){.15}
\pscircle*(25.24,1.94){.15}\pscircle*(27.36,1.94){.15}\pscircle*(30.36,1.94){.15}
\rput(24.9,4.7){\sm{$1^+$}}\rput(27.9,4.7){\sm{$2^+$}}\rput(30.9,4.7){\sm{$3^+$}}
\rput(24.9,1.4){\sm{$1^-$}}\rput(27.9,1.4){\sm{$2^-$}}\rput(30.9,1.4){\sm{$3^-$}}
\rput(27.8,0){$\ep(\cS)\!\cong_4\!\!2,~\cap\!\Ups_{0,3}$}
\pscircle[linewidth=.05](35,3){1.5}\pscircle[linewidth=.05](38,3){1.5}
\rput(32.5,2.5){\begin{large}$-2$\end{large}}
\pscircle*(36.5,3){.15}
\pscircle*(33.94,4.06){.15}\pscircle*(36.06,4.06){.15}\pscircle*(39.06,4.06){.15}
\pscircle*(33.94,1.94){.15}\pscircle*(36.06,1.94){.15}\pscircle*(39.06,1.94){.15}
\rput(33.6,4.7){\sm{$1^+$}}\rput(36.6,4.7){\sm{$3^+$}}\rput(39.6,4.7){\sm{$2^+$}}
\rput(33.96,1.4){\sm{$1^-$}}\rput(36.6,1.4){\sm{$3^-$}}\rput(39.6,1.4){\sm{$2^-$}}
\rput(36.5,0){$\ep(\cS)\!\cong_4\!\!2,~\cap\!\Ups_{0,3}$}
\end{pspicture}
\caption{The left-hand sides of the two identities represent
codimension~2 cycles with co-orientations in~$\R\ov\cM_{0,1,2}$ and~$\R\ov\cM_{0,0,3}$
that bound co-oriented hypersurfaces~$\Ups_{1,2}$ and~$\Ups_{0,3}$, respectively,
and their lifts to $\ov\fM_{0,k,l}^{\phi}(B;J)$.
The right-hand sides represent the corrections to the associated lifted relations 
from the crossings of the lifts of~$\Ups_{1,2}$ and~$\Ups_{0,3}$
with the codimension~1 strata~$\cS$ that obstruct the relative orientability 
of~\eref{evdfn_e}.
The involutions in all cases are the reflections around the horizontal center line.}
\label{TopolRel_fig2}
\end{figure}

If $k\!\ge\!1$, the $(k',l')\!=\!(1,2),(0,3)$ cases of the morphisms~\eref{Rffdfn_e}
are generally not relatively orientable.
The equality of the left- and right-hand sides of~\eref{RM12psirel_e} 
as co-oriented cycles in~$\R\ov\cM_{0,1,2}$ and 
the equality of the left- and right-hand sides of~\eref{RM03psirel_e}
as co-oriented cycles $\R\ov\cM_{0,0,3}$
are lifted in~\cite{RealWDVV,RealWDVVdim3} over these morphisms {\it along with} 
co-oriented bordisms
\hbox{$\Ups_{1,2}\!\subset\!\R\ov\cM_{0,1,2}$} and
\hbox{$\Ups_{0,3}\!\subset\!\R\ov\cM_{0,0,3}$} between the two sides. 
The intersections of the lifts of~$\Ups_{1,2}$ and~$\Ups_{0,3}$
with the codimension~1 strata~$\cS$ that obstruct the relative orientability 
of~\eref{evdfn_e} determine corrections to lifting the relations from 
$\R\ov\cM_{0,1,2}$ and $\R\ov\cM_{0,0,3}$ directly.
All such codimension~1 strata~$\cS$ consist of $J$-holomorphic maps 
from a pair of real components joined at a pair of real points 
which satisfy a certain numerical condition $\ep(\cS)\!\cong_4\!\!2$ 
as on the right-hand sides of the two equalities in Figure~\ref{TopolRel_fig2}.
The corrected lifted relations in turn translate to relations between counts 
of real three- and two-component genus~0 degree~$B$ $J$-holomorphic curves 
as in Figure~\ref{TopolRel_fig2} passing  through $k$~points in~$X^{\phi}$ 
and $l$~cycles in~$X$.
Splitting the nodal counts into counts of irreducible curves yields 
relations between
the real genus~0 Gromov-Witten invariants of~$(X,\om,\phi)$ and 
the standard genus~0 Gromov-Witten invariants of~$(X,\om)$.
The collection of these relations taken over all~$B$ and possible choices
of the constraints is equivalent to the PDEs~\eref{WDVVodeM12_e} and~\eref{WDVVodeM03_e}.\\

The open/real WDVV-type relations in dimension~2 obtained in~\cite{RealWDVV}
have the same structure as predicted in~\cite{Jake2}, but differ by crucial signs.
The former lead to the recursions for Welshinger's invariants predicted in~\cite{Jake2}
due to the comparison between the invariants of~\cite{Wel4} and~\cite{Jake}
predicted in~\cite{Jake2} also being off by signs.
The approach to the proof of real WDVV-type relations proposed in~\cite{Jake2}
is analogous in spirit to~\cite{RT} and 
involves defining a count of real curves for each element of $\R\ov\cM_{0,1,2}$
and $\R\ov\cM_{0,0,3}$, 
showing that these counts depend only on the homology classes of the insertions,
and are independent of the choice of the element in each of the two moduli spaces.
The desired invariance in fact does not hold in the case of~$\R\ov\cM_{0,1,2}$.
The proposed approach itself is more rigid than the lifting-of-cobordisms 
approach introduced in~\cite{RealWDVV} and does not connect as readily 
with the topological recursion relations.\\

Let $(X,\om,\phi)$ be a compact real symplectic sixfold with a Spin-structure on~$X^{\phi}$.
The real genus~0 Gromov-Witten invariants~\eref{RGWs_e} of~$(X,\om,\phi)$ appearing
in~\cite{RealWDVVdim3} differ from those in~\cite{Jake} by uniform multiples 
(dependent only on the number~$l$ of conjugate pairs of marked points) and
vanish on insertions from $H^2(X;\Q)^{\phi}_+$ and~$H^4(X;\Q)^{\phi}_-$. 
The $k\!=\!0$ invariants~\eref{RGWs_e} defined in~\cite{Penka2} agree with 
the $k\!=\!0$ invariants appearing in~\cite{RealWDVVdim3} when the insertions are taken
in $H^0(X;\Q)$, $H^2(X;\Q)$,  and $H^6(X;\Q)$, but vanish on $H^4(X;\Q)^{\phi}_+$.
In particular, the invariants of~\cite{Penka2} in the case of real symplectic sixfolds
with insertions in~$H^4(X;\Q)$ are not the same 
as the $k\!=\!0$ cases of the invariants arising from~\cite{Jake,RealWDVVdim3}.

\section{Complex and real WDVV PDEs}
\label{WDVV_subs}
 
Given a compact oriented even-dimensional manifold $X$,
we choose a basis $\mu_1^{\st},\ldots,\mu_N^{\st}$ for $H^{2*}(X;\Q)$
consisting of homogeneous elements.
Let $(g_{ij})_{i,j}$ be the $N\!\times\!N$-matrix given~by
$$g_{ij}=\blr{\mu_i^{\st}\mu_j^{\st},[X]}$$
and $(g^{ij})_{ij}$ be its inverse.
Thus, the Poincare dual of the diagonal $\De_X\!\subset\!X^2$ is given~by
\BE{diagsplit_e} \PD_{X^2}\big(\De_X)=\sum_{i,j\in[N]}\!\!g^{ij}\mu_i^{\st}\!\times\!\mu_j^{\st}
\in H^2(X^2;\Q)\big/\big(H^{\odd}(X;\Q)\!\times\!H^{\odd}(X;\Q)\!\big);\EE
see Theorem~11.11 in~\cite{MiSt}.
For a tuple $\bt\!\equiv\!(t_1,\ldots,t_N)$ of formal variables, define
$$\mu_{\bt}^{\st}=\mu_1^{\st}t_1\!+\!\ldots\!+\!\mu_N^{\st}t_N\,.$$

\vspace{.2in}

Suppose in addition $\om$ is a symplectic form on~$X$.
Let
$$\La_{\om}= \big\{(\Psi\!:H_2(X;\Z)\!\lra\!\Q)\!: \big|
\{B\!\in\!H_2(X;\Z)\!: \Psi(B)\!\neq\!0,\,\om(B)\!<\!E\}\big|\!<\!\i~\forall\,E\!\in\!\R\big\}.$$
We write an element $\Psi$ of $\La_{\om}$ as 
$$\Psi=\sum_{B\in H_2(X;\Z)}\hspace{-.2in}\Psi(B)q^B$$
and multiply two such elements as powers series in $q$ with the exponents in~$H_2(X;\Z)$.
For \hbox{$B\!\in\!H_2(X;\Z)$}, we denote~by
$$\lr{\ldots}_{\!B}^{\om}\!: \bigoplus_{l=0}^{\i}H^{2*}(X;\Q)^{\oplus l}\lra \Q$$
the genus~0 degree~$B$ Gromov-Witten invariants of~$(X,\om)$ and extend them linearly over 
the formal variables~$t_i$ above.
Define $\Phi_{\om}\!\in\!\La_{\om}[[t_1,\ldots,t_N]]$ by
\BE{CPhidfn_e}\Phi_{\om}(t_1,\ldots,t_N)=
\sum_{\begin{subarray}{c}B\in H_2(X;\Z)\\ l\in\Z^{\ge0}\end{subarray}}\!\!\!
\blr{\underset{l}{\underbrace{\mu_{\bt}^{\st},\ldots,\mu_{\bt}^{\st}}}}_{\!B}^{\om}\!\Bigg)
\frac{q^B}{l!}\,.\EE
By Gromov's Compactness Theorem,
the coefficients of the powers of $t_1,\ldots,t_N$ in $\Phi_{\om}$ lie in~$\La_{\om}$.
The relations on the genus~0 Gromov-Witten invariants of~$(X,\om)$ obtained 
by lifting the homology relation on~$\ov\cM_{0,4}$ represented by the top line
in Figure~\ref{TopolRel_fig} is equivalent to the set of \sf{WDVV PDEs}
\BE{CWDVV_e}  
\sum_{1\le i,j\le N}\!\!\!\!\!\!\!
\big(\prt_{t_a}\prt_{t_b}\prt_{t_i}\Phi_{\om}\big)g^{ij}
\big(\prt_{t_j}\prt_{t_c}\prt_{t_d}\Phi_{\om}\big)
=\sum_{1\le i,j\le N}\!\!\!\!\!\!\!
\big(\prt_{t_a}\prt_{t_c}\prt_{t_i}\Phi_{\om}\big)g^{ij}
\big(\prt_{t_j}\prt_{t_b}\prt_{t_d}\Phi_{\om}\big)\EE
with \hbox{$a,b,c,d\!=\!1,\ldots,N$}.\\

Let $\phi$ be an anti-symplectic involution on~$(X,\om)$.
Define 
$$\fd\!:H_2(X)\lra H_2(X)_-^{\phi}, \qquad \fd(B)=B\!-\!\phi_*(B)\,.$$
For a (relatively) Spin or Pin-structure~$\fs$ on (the tangent bundle) of~$X^{\phi}$,
$B\!\in\!H_2(X;\Z)$, and $k\!\in\!\Z^{\ge0}$, we denote~by
\BE{RGWs_e} \lr{\ldots}_{\!\fs;B;k}^{\om;\phi}\!: \bigoplus_{l=0}^{\i}H^{2*}(X;\Q)^{\oplus l}
\lra \Q\EE
the real genus~0 degree~$B$ Gromov-Witten invariants of~$(X,\om,\phi)$ 
with $k$~real marked points associated with~$\fs$, whenever they are defined.
We extend them linearly over the formal variables~$t_i$ above.
Define $\Phi_{\om}^{\phi}\!\in\!\La_{\om}[[t_1,\ldots,t_N]]$
and $\Om_{\om;\fs}^{\phi}\!\in\!\La_{\om}[[t_1,\ldots,t_N,u]]$ by
\begin{gather*}
\Phi_{\om}^{\phi}(t_1,\ldots,t_N)=
\sum_{\begin{subarray}{c}B\in H_2(X;\Z)\\ l\in\Z^{\ge0}\end{subarray}}\!\!\!
\Bigg(\sum_{\begin{subarray}{c}B'\in H_2(X;\Z)\\
\fd(B')=B\end{subarray}}\!\!\!\!\!\blr{
\underset{l}{\underbrace{\mu_{\bt}^{\st},\ldots,\mu_{\bt}^{\st}}}}_{\!B'}^{\om}\!\Bigg)
\frac{q^B}{l!}\,,\\
\Om_{\om;\fs}^{\phi}(t_1,\ldots,t_N,u)=
\sum_{\begin{subarray}{c}B\in H_2(X;\Z)\\ k,l\in\Z^{\ge0}\end{subarray}}\!\!\!
\blr{
\underset{l}{\underbrace{\mu_{\bt}^{\st},\ldots,\mu_{\bt}^{\st}}}}_{\!\fs;B,k}^{\!\om;\phi}
\frac{2^{1-l}q^Bu^k}{k!l!}\,.
\end{gather*}
By Gromov's Compactness Theorem and the assumption that $\phi^*\om\!=\!-\om$,
the inner sum in the definition of~$\Phi_{\om}^{\phi}$ has finitely nonzero terms.
For the same reason, the coefficients of the powers of $t_1,\ldots,t_N,u$ in
$\Phi_{\om}^{\phi}$ and $\Om_{\om;\fs}^{\phi}$ lie in~$\La_{\om}$.\\

Alcolado's thesis~\cite{Adam} predicts that the real genus~0 Gromov-Witten invariants
of~$(X,\om,\phi)$ satisfy the set of \sf{extended WDVV PDEs} 
\begin{gather}
\label{WDVVodeM12_e} 
\sum_{1\le i,j\le N}\!\!\!\!\!\!\!\big(\prt_{t_a}\prt_{t_b}\prt_{t_i}\Phi_{\om}^{\phi}\big)g^{ij}
\big(\prt_{t_j}\prt_u\Om_{\om;\fs}^{\phi}\big)+
\big(\prt_{t_a}\prt_{t_b}\Om_{\om;\fs}^{\phi}\big)\!
\big(\prt_u^2\Om_{\om;\fs}^{\phi}\big)
=\big(\prt_{t_a}\prt_u\Om_{\om;\fs}^{\phi}\big)\!
\big(\prt_{t_b}\prt_u\Om_{\om;\fs}^{\phi}\big),\\
\label{WDVVodeM03_e} 
\begin{split}
&\sum_{1\le i,j\le N}\!\!\!\!\!\!\!\big(\prt_{t_a}\prt_{t_b}\prt_{t_i}\Phi_{\om}^{\phi}\big)g^{ij}
\big(\prt_{t_j}\prt_{t_c}\Om_{\om;\fs}^{\phi}\big)+
\big(\prt_{t_a}\prt_{t_b}\Om_{\om;\fs}^{\phi}\big)\!
\big(\prt_{t_c}\prt_u\Om_{\om;\fs}^{\phi}\big)\\
&\hspace{1.4in}=
\sum_{1\le i,j\le N}\!\!\!\!\!\!\!\big(\prt_{t_a}\prt_{t_c}\prt_{t_i}\Phi_{\om}^{\phi}\big)g^{ij}
\big(\prt_{t_j}\prt_{t_b}\Om_{\om;\fs}^{\phi}\big)+
\big(\prt_{t_a}\prt_{t_c}\Om_{\om;\fs}^{\phi}\big)\!
\big(\prt_{t_b}\prt_u\Om_{\om;\fs}^{\phi}\big)
\end{split}\end{gather}
with \hbox{$a,b,c\!=\!1,\ldots,N$}, with appropriate definitions of the real invariants.
Alcolado's predictions have been confirmed
\begin{enumerate}[label=(C\arabic*),leftmargin=*]

\item\label{Rcond_it1} in~\cite{RealWDVV} if 
$(X,\om,\phi)$ is a real symplectic fourfold,
$\fs$ is a Pin$^-$-structure on~$X^{\phi}$,
and the real invariants are defined as in~\cite{RealWDVV};

\item\label{Rcond_it4} in~\cite{RealWDVVdim3} if 
$(X,\om,\phi)$ is a real symplectic sixfold with a finite-order automorphism 
which restricts to an orientation-reversing diffeomorphism of~$X^{\phi}$ and
acts trivially on the $(-\phi)_*$-invariant part of~$H_2(X,Y;\Z)$, 
$\fs$ is a Spin-structure on~$X^{\phi}$,
and the real invariants are defined as in~\cite{RealWDVVdim3}.\\

\end{enumerate}

A specialization of~\eref{WDVVodeM03_e} is obtained in~\cite{RealEnum}
for a real symplectic $2n$-fold $(X,\om,\phi)$ with $n$~odd
which admits a complex vector~$E$ with a conjugation~$\wt\phi$ lifting~$\phi$
so~that
\BE{PenkaCond_e}\begin{split}
&\hspace{1.5in}
w_1(E^{\wt\phi})^2=w_2(TX^{\phi}) \qquad\hbox{and}\\
&\blr{2\mu(E,E^{\wt\phi})\!+\!\mu(X,X^{\phi}),\be}\in4\Z\quad
\forall~\be\!\in\!H_2(X,X^{\phi};\Z)~\hbox{with}~\phi_*\be=-\be,
\end{split}\EE
where $E^{\wt\phi}\!\lra\!X^{\phi}$ is the fixed locus of the $\wt\phi$-action on~$E$
and $\mu(X,X^{\phi})$ and $\mu(E,E^{\wt\phi})$ are the Maslov indices of the pairs
$(TX,TX^{\phi})$ and~$(E,E^{\wt\phi})$, respectively.
A Spin-structure~$\fs$ on $2E^{\wt\phi}\!\oplus\!TX^{\phi}$
(along with some additional data if the equality 
on the second line of~\eref{PenkaCond_e} does not hold for all $\be\!\in\!H_2(X,X^{\phi};\Z)$)
determines real genus~0 Gromov-Witten invariants~\eref{RGWs_e} with $k\!=\!0$;
see Theorem~1.6 in~\cite{Penka2}.
We formally define the functionals~\eref{RGWs_e} with $k\!\ge\!1$ to be zero.
Suppose $\mu_1^{\st},\ldots,\mu_{N^-}^{\st}$ and $\mu_{N_-+1}^{\st},\ldots,\mu_N^{\st}$
are bases for~$H^{2*}(X;\Q)^{\phi}_-$ and~$H^{2*}(X;\Q)^{\phi}_+$, respectively.
Theorem~2.1 in~\cite{RealEnum} is then equivalent to
the PDEs~\eref{WDVVodeM03_e} with $a\!\in\![N]\!-\![N_-]$ and $b,c\!\in\![N_-]$.
By Theorem~2.2 in~\cite{RealEnum}, $\prt_{t_a}\prt_{t_b}\Om_{\om;\fs}^{\phi}$
and $\prt_{t_a}\prt_{t_c}\Om_{\om;\fs}^{\phi}$ in~\eref{WDVVodeM03_e} vanish 
with the definitions in~\cite{Penka2} because $\mu_{t_a}^{\st}\!\in\!H^{2*}(X;\Q)^{\phi}_+$.\\

Motivated by~\cite{Fuk10,Fuk11}, Solomon and Tukachinsky~\cite{JS12} 
use bounding chains in Fukaya's $A_{\i}$-algebras to define counts 
of $J$-holomorphic disks in a compact symplectic manifold~$(X,\om)$
with boundary in a compact Lagrangian~$Y$ with a relative Spin-structure~$\fs$.
According to Theorem~3 in~\cite{JS3}, these counts satisfy
the real WDVV PDEs~\eref{WDVVodeM12_e} and~\eref{WDVVodeM03_e}
if the generating function~$\Om_{\om;\fs}^{\phi}$ is defined without the factors of~$2^{1-l}$. 
As noted in Remark~2.7 of~\cite{RealEnum} and Theorem~2.10 of~\cite{JakeSaraGeom}, 
the use of a relative Spin structure~$\fs$ to orient the moduli spaces of disks
leads to the necessity to twist the generating function~$\Phi_{\om}^{\phi}$
by $(-1)^{\lr{w_2(\fs),B'}}$ for the purposes of~\eref{WDVVodeM12_e} and~\eref{WDVVodeM03_e}
(the authors of~\cite{JS12,JS3} are currently re-adjusting their signs to account for~this). 
This twist can be avoided if there exists a complex vector bundle $E\!\lra\!X$ and 
a totally real subbundle $E^{\R}\!\lra\!Y$ of~$E|_Y$ of maximal rank so~that
$w_1(E^{\R})^2\!=\!w_2(TY)$.
A Spin-structure~$\fs$ on $2E^{\R}\!\oplus\!TY$ then determines orientations
on the moduli spaces of disk maps to~$(X,Y)$ as in Section~7 in~\cite{Penka2}.
If the construction of~\cite{JS12} were carried out with these orientations,
instead of those induced by relative Spin structures as in Section~8.1 of~\cite{FOOO}, 
then the PDEs~\eref{WDVVodeM12_e} and~\eref{WDVVodeM03_e}  
would hold in~\cite{JS3} without the above sign twist.
The degree~$m$ holomorphic line \hbox{$E\!\equiv\!\cO_{\P^{2m-1}}(m)$} over~$\P^{2m-1}$
with the standard conjugation~$\wt\phi$ satisfies~\eref{PenkaCond_e}
and can thus be used to orient the moduli space of stable disk and real maps 
as in~\cite{Penka2}.
Theorem~10 in~\cite{JS3}, stating the PDEs~\eref{WDVVodeM12_e} and~\eref{WDVVodeM03_e}
for $(\P^{2m-1},\tau_{2m-1})$, requires the use of the resulting orientations,
rather than those induced by the associated relative Spin-structures.
By Theorem~1.4 in~\cite{RealGWsII}, these orientations on the main stratum 
$\fM_{0,k,l}^{\tau_{2m-1}}(d;J)$ of the moduli space of stable real degree~$d$
holomorphic maps to~$(\P^{2m-1},\tau_{2m-1})$  are different if and only~if
$m\!\not\in\!2\Z$ and $md\!\cong\!1,2$ mod~4.

\section{Relation between the PDEs~\eref{WDVVodeM12_e} and~\eref{WDVVodeM03_e}}
\label{PDE_sec}

We now show that the PDEs~\eref{WDVVodeM12_e} often imply the PDEs~\eref{WDVVodeM03_e}.

\begin{thm}\label{RWDVVimpl_thm}
If the functional~\eref{RGWs_e} is nonzero for some $k\!\ge\!2$,
then the relations~\eref{WDVVodeM12_e} imply the relations~\eref{WDVVodeM03_e}.
\end{thm}

\begin{proof}
The WDVV PDE~\eref{CWDVV_e} for the Gromov-Witten invariants of~$(X,\om)$ gives
\BE{CWDDVV_e}
\sum_{1\le l,m\le N}\!\!\!\!\!\!\!
\big(\prt_{t_a}\prt_{t_b}\prt_{t_l}\Phi_{\om}^{\phi}\big)g^{lm}\!
\big(\prt_{t_m}\prt_{t_c}\prt_{t_i}\Phi_{\om}^{\phi}\big)
=\sum_{1\le l,m\le N}\!\!\!\!\!\!\!
\big(\prt_{t_a}\prt_{t_c}\prt_{t_l}\Phi_{\om}^{\phi}\big)g^{lm}\!
\big(\prt_{t_m}\prt_{t_b}\prt_{t_i}\Phi_{\om}^{\phi}\big)\EE
for all $a,b,c,i\!=\!1,\ldots,N$.
Multiplying~\eref{WDVVodeM12_e} by $\prt_{t_c}\prt_u\Om_{\om;\fs}^{\phi}$ gives
\BE{WDVVodeM12a_e}\begin{split}
\big(\prt_u^2\Om_{\om;\fs}^{\phi}\big)\!\big(\prt_{t_a}\prt_{t_b}\Om_{\om;\fs}^{\phi}\big)\!
\big(\prt_{t_c}\prt_u\Om_{\om;\fs}^{\phi}\big)
&=\big(\prt_{t_a}\prt_u\Om_{\om;\fs}^{\phi}\big)\!
\big(\prt_{t_b}\prt_u\Om_{\om;\fs}^{\phi}\big)\!
\big(\prt_{t_c}\prt_u\Om_{\om;\fs}^{\phi}\big)\\
&\hspace{.4in}-\big(\prt_{t_c}\prt_u\Om_{\om;\fs}^{\phi}\big)\!\!\!\!\!
\sum_{1\le i,j\le N}\!\!\!\!\!\!\big(\prt_{t_a}\prt_{t_b}\prt_{t_i}\Phi_{\om}^{\phi}\big)g^{ij}
\!\big(\prt_{t_j}\prt_u\Om_{\om;\fs}^{\phi}\big)\,.
\end{split}\EE
Replacing $(a,b)$ in~\eref{WDVVodeM12_e} with $(c,m)$
and multiplying the resulting equation by 
$\prt_{t_a}\prt_{t_b}\prt_{t_l}\Phi_{\om}^{\phi}g^{lm}$ with $c,l,m\!=\!1,\ldots,N$,
we obtain
\BE{WDVVodeM12b_e}\begin{split}
\big(\prt_u^2\Om_{\om;\fs}^{\phi}\big)\!
\big(\prt_{t_a}\prt_{t_b}\prt_{t_l}\Phi_{\om}^{\phi}\big)g^{lm}
\big(\prt_{t_m}\prt_{t_c}\Om_{\om;\fs}^{\phi}\big)
&=\big(\prt_{t_c}\prt_u\Om_{\om;\fs}^{\phi}\big)\!
\big(\prt_{t_a}\prt_{t_b}\prt_{t_l}\Phi_{\om}^{\phi}\big)g^{lm}
\big(\prt_{t_m}\prt_u\Om_{\om;\fs}^{\phi}\big)\\
&\quad-
\big(\prt_{t_a}\prt_{t_b}\prt_{t_l}\Phi_{\om}^{\phi}\big)g^{lm}\!\!\!\!\!\!
\sum_{1\le i,j\le N}\!\!\!\!\!\!\big(\prt_{t_c}\prt_{t_m}\prt_{t_i}\Phi_{\om}^{\phi}\big)g^{ij}
\!\big(\prt_{t_j}\prt_u\Om_{\om;\fs}^{\phi}\big)\,.
\end{split}\EE
Summing up~\eref{WDVVodeM12b_e} over all $l,m\!=\!1,\ldots,N$ and adding~\eref{WDVVodeM12a_e}
to the result, we find~that
\BE{WDVVodeM12c_e}\begin{split}
\big(\prt_u^2\Om_{\om;\fs}^{\phi}\big)\!
\big(\tn{LHS of~\eref{WDVVodeM03_e}}\!\big)&=\big(\prt_{t_a}\prt_u\Om_{\om;\fs}^{\phi}\big)\!
\big(\prt_{t_b}\prt_u\Om_{\om;\fs}^{\phi}\big)\!
\big(\prt_{t_c}\prt_u\Om_{\om;\fs}^{\phi}\big)\\
&-\sum_{1\le i,j\le N}\!\!\Bigg(\!\sum_{1\le l,m\le N}\!\!\!\!\!\!\!
\big(\prt_{t_a}\prt_{t_b}\prt_{t_l}\Phi_{\om}^{\phi}\big)g^{lm}\!
\big(\prt_{t_m}\prt_{t_c}\prt_{t_i}\Phi_{\om}^{\phi}\big)\!\!\!\Bigg)\!g^{ij}
\!\big(\prt_{t_j}\prt_u\Om_{\om;\fs}^{\phi}\big)\,.
\end{split}\EE
The right-hand side of~\eref{WDVVodeM03_e} times $\prt_u^2\Om_{\om;\fs}^{\phi}$ is 
given by the right-hand side of~\eref{WDVVodeM12c_e} with the roles of~$b$ and~$c$ 
interchanged.
Thus,
\begin{equation*}\begin{split}
&\big(\prt_u^2\Om_{\om;\fs}^{\phi}\big)\!\Big(\!\!
\big(\tn{LHS of~\eref{WDVVodeM03_e}}\!\big)-
\big(\tn{RHS of~\eref{WDVVodeM03_e}}\!\big)\!\!\Big)\\
&=\sum_{1\le i,j\le N}\!\!\Bigg(\!\sum_{1\le l,m\le N}\!\!\!\!\!\!\!
\Big(\!\!\big(\prt_{t_a}\prt_{t_c}\prt_{t_l}\Phi_{\om}^{\phi}\big)g^{lm}\!
\big(\prt_{t_m}\prt_{t_b}\prt_{t_i}\Phi_{\om}^{\phi}\big)
-\big(\prt_{t_a}\prt_{t_b}\prt_{t_l}\Phi_{\om}^{\phi}\big)g^{lm}\!
\big(\prt_{t_m}\prt_{t_c}\prt_{t_i}\Phi_{\om}^{\phi}\big)\!\!\Big)
\!\!\!\Bigg)\!g^{ij}
\!\big(\prt_{t_j}\prt_u\Om_{\om;\fs}^{\phi}\big).
\end{split}\end{equation*}
Combing this statement with~\eref{CWDDVV_e}, we conclude that
\BE{WDVVodeM03b_e}
\big(\prt_u^2\Om_{\om;\fs}^{\phi}\big)\!\Big(\!\!
\big(\tn{LHS of~\eref{WDVVodeM03_e}}\!\big)-
\big(\tn{RHS of~\eref{WDVVodeM03_e}}\!\big)\!\!\Big)=0\in
\La_{\om}[[t_1,\ldots,t_N,u]].\EE
This establishes the claim.
\end{proof}

\section{Proof of Theorem~\ref{Cpsi_thm}}
\label{Crelpf_sec}

By symmetry, it is sufficient to establish this proposition for $i\!=\!2$ and $j\!=\!3$.
Since $\ov\cM_{0,3}$ is a single point, Theorem~\ref{Cpsi_thm} holds for $l\!=\!3$.\\

Let $l\!\ge\!3$ and 
$$\ff\!: \ov\cM_{0,l+1}\lra\ov\cM_{0,l}$$
be the forgetful morphism dropping the last marked point.
For each partition $I\!\cup\!J$ of $[l]$ with $|I|,|J|\!\ge\!2$, 
the restriction of~$\ff$ to the complement of $D_{I(l+1),J}\!\cap\!D_{I,J(l+1)}$
in $\ov\cM_{0,l+1}$ is transverse to~$D_{I,J}$.
Since
\begin{gather*}
\ff^{-1}\big(D_{I,J}\big)= D_{I(l+1),J}\!\cup\!D_{I,J(l+1)}, \\ 
\big\{\nd\ff|_{D_{I(l+1),J}-D_{I,J(l+1)}}\big\}^*\fo_{D_{I,J}}^c
=\fo_{D_{I(l+1),J}}^c, \quad
\big\{\nd\ff|_{D_{I,J(l+1)}-D_{I(l+1),J}}\big\}^*\fo_{D_{I,J}}^c
=\fo_{D_{I,J(l+1)}}^c,
\end{gather*}
and $D_{I(l+1),J}\!\cap\!D_{I,J(l+1)}$ is of real codimension~2 in 
$D_{I(l+1),J}$ and~$D_{I,J(l+1)}$,
\BE{Crelpf_e3}
\ff^*\big(\PD_{\ov\cM_{0,l}}\!\big(D_{I,J}\big)\big)
=\PD_{\ov\cM_{0,l+1}}\!\big(D_{I(l+1),J}\big)\!+\!
\PD_{\ov\cM_{0,l+1}}\!\big(D_{I,J(l+1)}\big)
\in H^2\big(\ov\cM_{0,l+1};\Z\big)\,.\EE
On the other hand, the section 
$$s\in\Ga\big(\ov\cM_{0,l+1};L_1^*\!\otimes\!\ff^*L_1\big),
\qquad s(w)=\nd\ff(w),$$
vanishes only along the divisor $D_{1(l+1),[l]-\{1\}}$.
Since $\nd s$ induces an orientation-preserving isomorphism from 
$(\cN D_{1(l+1),[l]-\{1\}},\fo_{D_{1(l+1),[l]-\{1\}}}^c)$ to 
$L_1^*\!\otimes\!\ff^*L_1|_{D_{1(l+1),[l]-\{1\}}}$,
we conclude~that 
\BE{Crelpf_e5} \psi_1=\ff^*\psi_1\!+\!\PD_{\ov\cM_{0,l+1}}\!\big(D_{1(l+1),[l]-\{1\}}\big)
\in H^2\big(\ov\cM_{0,l+1};\Z\big).\EE
Combining \eref{Crelpf_e3} with \eref{Crelpf_e5}, we obtain
the identity of Theorem~\ref{Cpsi_thm} by induction.

\section{Proof of Theorem~\ref{Rpsi_thm}}
\label{Rrelpf_sec}

By symmetry, it is sufficient to assume that $i\!=\!2$ in the second statement.
Since $\R\ov\cM_{0,1,1}$ is a single point and $\R\ov\cM_{0,0,2}$ is a circle,
the first statement of this proposition holds for $(k,l)\!=\!(1,1)$
and the second for $(k,l)\!=\!(0,2)$.\\

If $k\!+\!2l\!\ge\!3$, let   
$$\ff\!: \R\ov\cM_{0,k,l+1}\lra\R\ov\cM_{0,k,l}
\qquad\hbox{and}\qquad
\ff_{\R}\!: \R\ov\cM_{0,k+1,l}\lra\R\ov\cM_{0,k,l}$$
be the forgetful morphisms dropping the last conjugate pair of marked points and
the last real marked point, respectively.
By the same reasoning as for~\eref{Crelpf_e3},
\BE{Rrelpf_e3}\begin{split}
\ff^*\big(\PD_{\R\ov\cM_{0,k,l}}\!\big(\R D_{I,J},\fo_{\R D_{I,J}}^c\big)\big)
=\,&\PD_{\R\ov\cM_{0,k,l+1}}\!\big(\R D_{I(l+1),J},\fo_{\R D_{I(l+1),J}}^c\big)\\
&+\PD_{\R\ov\cM_{0,k,l+1}}\!\big(\R D_{I,J(l+1)},\fo_{\R D_{I,J(l+1)}}^c\big)
\in H^2\big(\R\ov\cM_{0,k,l+1};\Z\big)
\end{split}\EE
for each partition $I\!\cup\!J$ of $[l]$ with $|I|\!\ge\!2$ and either $J\!\neq\!\eset$
or $k\!>\!0$.
We also~have 
\BE{Rrelpf_e3b}
\ff_{\R}^*\big(\PD_{\R\ov\cM_{0,k,l}}\!\big(\R D_{I,J},\fo_{\R D_{I,J}}^c\big)\big)
=\PD_{\R\ov\cM_{0,k+1,l}}\!\big(\R D_{I,J},\fo_{\R D_{I,J}}^c\big)
\in H^2\big(\R\ov\cM_{0,k+1,l};\Z\big)\EE
for each partition $I\!\cup\!J$ of $[l]$ with $|I|\!\ge\!2$ and either $J\!\neq\!\eset$
or $k\!>\!0$.
By the same reasoning as for~\eref{Crelpf_e5},
\BE{Rrelpf_e5}\begin{split}
\psi_1=\ff^*\psi_1\!+\!\PD_{\R\ov\cM_{0,k,l+1}}\!\big(\R D_{1(l+1),[l]-\{1\}},
\fo_{\R D_{1(l+1),[l]-\{1\}}}^c\big)&\in H^2\big(\R\ov\cM_{0,k,l+1};\Z\big),\\
\psi_1=\ff_{\R}^*\psi_1\in H^2\big(\R\ov\cM_{0,k+1,l};\Z\big)&.
\end{split}\EE
Combining \eref{Rrelpf_e3} and \eref{Rrelpf_e3b} with \eref{Rrelpf_e5}, we obtain
the two identities of Theorem~\ref{Rpsi_thm} by induction.\\

\vspace{.3in}

{\it Department of Mathematics, Stony Brook University, Stony Brook, NY 11794\\
azinger@math.stonybrook.edu}

\end{document}